\documentclass{article}
\usepackage{amsfonts,amsmath,amsthm,amssymb}
\usepackage{graphics, epsfig}
\usepackage{color}
\usepackage{appendix}
\usepackage{ulem}
\usepackage[makeroom]{cancel}
\usepackage{fancyhdr}
\usepackage{centernot}
\usepackage{mathtools}
\usepackage{ stmaryrd }

 \usepackage[usenames,dvipsnames]{pstricks}
 \usepackage{pst-grad} % For gradients
 \usepackage{pst-plot} % For axes
\allowdisplaybreaks

\let\TeXchi\chi
\newbox\chibox
\setbox0 \hbox{\mathsurround0pt $\TeXchi$}
\setbox\chibox \hbox{\raise\dp0 \box 0 }
\def\chi{\copy\chibox}

\newtheorem{proposition}{Proposition}[section]
\newtheorem{theorem}{Theorem}[section]

\newtheorem{remark}{Remark}[section]
\numberwithin{equation}{section}
\numberwithin{theorem}{section}
\numberwithin{definition}{section}
\numberwithin{example}{section}
\numberwithin{proposition}{section}
\numberwithin{lemma}{section}
\numberwithin{remark}{section}
\DeclareMathOperator{\Ima}{Im}
\DeclareMathOperator{\Dom}{Dom}
\DeclareMathOperator{\pt}{pt}
\DeclareMathOperator{\Tot}{Tot}
\newcommand\blfootnote[1]{%
  \begingroup
  \renewcommand\thefootnote{}\footnote{#1}%
  \addtocounter{footnote}{-1}%
  \endgroup
}
\begin{document}
%%%%%%%%%%%%%%%%%%%%%%%%%%%%%%%%%%%%%%%%%%%%%%%%%%%
\title{Two cohomology theories for structured spaces}
%%%%%%%%%%%%%%%%%%%%%%%%%%%%%%%%%%%%%%%%%%%%%%%%%%%
\author
{Manuel Norman}
%%%%%%%%%%%%%%%%%%%%%%%%%%%%%%%%%%%%%%%%%%%%%%%%%%%
\date{}
\maketitle
%%%%%%%%%%%%%%%%%%%%%%%%%%%%%%%%%%%%%%%%%%%%%%%%%%%
\begin{abstract}
\noindent In [1] we defined a new kind of space called 'structured space' which locally resembles, near each of its points, some algebraic structure. We noted in the conclusion of the cited paper that the maps $f_s$ and $h$, which are of great importance in the theory of structured spaces, have some connections with the notions of presheaves (and hence also sheaves) and vector bundles. There are well known cohomology theories involving such objects; this suggests the possibility of the existence of (co)homology theories for structured spaces which are somehow related to $f_s$ and $h$. In this paper we indeed develop two cohomology theories for structured spaces: one of them arises from $f_s$, while the other one arises from $h$. In order to do this, we first develop a more general cohomology theory (called rectangular cohomology in the finite case, and square cohomology in the infinite case), which can actually be applied also in many other situations, and then we obtain the cohomology theories for structured spaces as simple consequences of this theory.
%%%%%%%%%%%%%%%%%%%%%%%%%%%%%%%%%%%%%%%%%%%%%%%%%%%
\end{abstract}
%%%%%%%%%%%%%%%%%%%%%%%%%%%%%%%%%%%%%%%%%%%%%%%%%%%
\blfootnote{Author: \textbf{Manuel Norman}; email: manuel.norman02@gmail.com\\
\textbf{AMS Subject Classification (2010)}: 55N35, 18G60\\
\textbf{Key Words}: structured space, cohomology, Leech cohomology}
%%%%%%%%%%%%%%%%%%%%%%%%%%%%%%%%%%%%%%%%%%%%%%%%%%%
\section{Introduction}
The aim of this paper is to develop two cohomology theories starting from $f_s$ and $h$. We refer to [1] for the basic theory of structured spaces (in particular, see Section 2 for the structure map $f_s$, and see the end of Section 4 for $h$). We noted in the conclusion of [1] that these two maps are similar, in some sense, to presheaves (and then also sheaves) and vector bundles. Indeed, when we deal with presheaves we have a functor which assigns to each open subset of a topological space $X$ a set of sections and a set of restrictions; when we deal with vector bundles, we assign to each point of the considered space $X$ a certain vector space, in such a way that some properties are satisfied (for a precise definition, we need the notion of fiber). We do not go into details (see [27-30] for these topics), since these concepts are not needed in the rest of the paper. They do suggest, however, a relation with $f_s$ and $h$, since we associate to each $U_p \in \mathcal{U}$ (or to each $x \in X$, if we use $\widehat{f}_s$, which is equivalent to $f_s$, as shown in Section 2.5 of [1]) its local structure via $f_s$, and we associate to each $x \in X$ all the local structures in which it is contained via $h$. In this first section we recall what homology and cohomology are, and we also review Leech cohomology. Then, in the next section we develop a general cohomology theory for a sequence of monoids, which will be applied in Section 3 to $f_s$ and in Section 4 to $h$.
\subsection{Homology and cohomology}
We refer to [17,18] and [31] for more details on these topics. Start with a chain complex of abelian groups (or modules) $A_p$ and of homomorphisms $d_p$:
$$... \xleftarrow{d_{-1}} A_0 \xleftarrow{d_0} A_1 \xleftarrow{d_1} A_2 \xleftarrow{d_2} ... $$
Suppose this chain stops on the left with the trivial group $(0)$ (meaning that $A_t = (0)$ $\forall t < 0$, and we only write one of these groups on the left):
$$ 0 \xleftarrow{d_{-1}} A_0 \xleftarrow{d_0} A_1 \xleftarrow{d_1} A_2 \xleftarrow{d_2} ...$$
(We note that this is equivalent to:
$$ ... \xrightarrow{d_2} A_2 \xrightarrow{d_1} A_1 \xrightarrow{d_0} a_0 \xrightarrow{d_{-1}} 0$$
the important thing is that the arrows (and the maps) are from $C_{p+1}$ to $C_p$).\\
A chain complex as above is called an \textit{exact sequence} if:
$$\ker d_p = \Ima d_{p+1}$$
$\forall p \geq 0$. A homology theory measures "how far" a chain complex is from being exact. To do this, we first require that (here $0$ is the identity element of $A_{p-1}$):
$$d_p \circ d_{p+1} = 0$$
which implies that
$$ \Ima d_{p+1} \subseteq \ker d_p $$
Then, this inclusion allows us to consider the quotient:
$$ H_p(A):= \ker d_p / \Ima d_{p+1} $$
where $(A,d)$ represents the chain above, $(A_{\bullet}, d_{\bullet})$.\\
These groups (or modules) $H_p$ give us a way to measure the imperfection we were talking about. They are called homology groups (or homology modules), and they constitute a homology theory for $(A,d)$.\\
Cohomology is obtained by dualising everything in the definitions above. A cochain complex:
$$ 0 \xrightarrow{d^{-1}} A_0 \xrightarrow{d^0} A_1 \xrightarrow{d^1} A_2 \xrightarrow{d^2} ... $$
(again, we notice that this is equivalent to say:
$$ ... \xleftarrow{d_{2}} A_2 \xleftarrow{d_1} A_1 \xleftarrow{d_0} A_0 \xleftarrow{d_{-1}} 0$$
the important thing is that the arrows (and the maps) are from $C_p$ to $C_{p+1}$) is exact if
$$ d^{p+1} \circ d^p = 0 $$
$\forall p \geq 0$. The cohomology groups (or modules):
$$H^p(A):= \ker d^{p+1} / \Ima d^p$$
measure how far a cochain complex is from being exact. They consistute a cohomology theory for $(A,d):=(A_{\bullet}, d^{\bullet})$.\\
There are various kinds of homology and cohomology theories; just to cite a few, we recall the well known de Rham cohomology, sheaf cohomology, \v{C}ech cohomology, K-theory, singular (co)homology, Borel-Moore homology, intersection homology, Hochschild (co)homology, ... arriving to prismatic cohomology, developed by Bhatt and Scholze in [19]. We refer to [20-26] and [43-45] for more on these topics. Some cohomology theories turn out to be equivalent in some particular situations. For some examples of these comparisons of cohomologies, see [33-35] and also [27,31].
\subsection{Leech cohomology}
We now review an important tool which will be used later. We will only consider the aspects needed for our cohomology theories, for more details we refer to [2-16]. The idea behind Leech cohomology is to have a cohomology theory for monoids. There are well known cohomology theories for groups, for instance, but we are also interested in cohomologies for some substructures. Leech cohomology is indeed a cohomology theory for monoids. The historical development of this cohomology theory is however different from the following exposition. This is because Grillet noticed that Leech cohomology can be reconduced to Barr-Beck cotriple cohomology (see [2,3]). We refer to the works by Grillet [10-13] for more details. The main references for what follows can be found in [4-7]. We start by considering a monoid $(M, \cdot)$. We assume the reader is familiar with cohomologies of small categories, even though this is not needed for the rest of the paper. We construct a category $\mathcal{D}M$ with object set $M$ and morphism set $M^3$, with $(a,b,c): b \rightarrow abc$. Composition is given by $(\widehat{a},abc,\widehat{c})(a,b,c,)=(\widehat{a}a,b,c \widehat{c})$, and the identity morphism of an object $a \in M$ is $id_a=(e,a,e)$, with $e$ identity element of $M$. Now consider a $\mathcal{D}M$-module $\mathcal{A}:\mathcal{D}M \rightarrow \textbf{Ab}$ (the category of abelian groups), defined by associating to each $a \in M$ an abelian group $\mathcal{A}(a)$, and by assigning to each morphism $(a,b,c)$ the group homomorphism $a_{*} c^{*}: \mathcal{A}(b) \rightarrow \mathcal{A}(abc)$. So we have, following [4], two families (one of abelian groups, and one of homomorphisms):
$$ (\mathcal{A}(a))_{a \in M}, \quad (\mathcal{A}(b) \xrightarrow{a_{*}} \mathcal{A}(ab) \xleftarrow{b^{*}} \mathcal{A}(a))_{a,b \in M} $$
satisfying the following relations ($\forall a,b,c \in M$):
$$ (ab)_{*}=a_{*} b_{*}, \quad c^{*} a_{*} = a_{*} c^{*}, \quad (bc)^{*}=c^{*} b^{*}, \quad e_{*}=e^{*}= id_{\mathcal{A}(a)} $$
This preparation allows us to define a cochain complex for $M$ with coefficients in $\mathcal{A}$, and also the cohomology groups associated to it. First, define the abelian groups:
\begin{multline}\label{Eq:1.1}
C^0 _L (M, \mathcal{A}):= \mathcal{A}(e)\\
C^n _L (M, \mathcal{A}) := \lbrace f \in \prod_{a_i \in M} \mathcal{A}(a_1 \cdot \cdot \cdot a_n) : f(a_1, ..., a_n)=0 \, \text{whenever} \, a_j=e \, \text{for some} \, j \rbrace, \, \text{for} \, n>0
\end{multline}
and the coboundary maps $\partial^n : C^n _L (M, \mathcal{A}) \rightarrow C^{n+1} _L (M, \mathcal{A})$  (we will use additive notation when dealing with $C^n _L (M, \mathcal{A})$):
\begin{multline}\label{Eq:1.2}
(\partial^0 f)(a):= a_{*}(f) - a^{*}(f)\\
(\partial^n f)(a_1, ..., a_{n+1}):= (a_1)_{*} f(a_2, ..., a_n) + \sum_{j=1}^{n} (-1)^j f(a_1, ..., a_j a_{j+1}, ..., a_{n+1}) +\\
+ (-1)^{n+1} (a_{n+1})^{*} f(a_1, ..., a_n), \, for \, n>0
\end{multline}
so that we have the cochain:
\begin{equation}\label{Eq:1.3}
0 \rightarrow C^0 _L (M, \mathcal{A}) \xrightarrow{\partial^0} C^1 _L (M, \mathcal{A}) \xrightarrow{\partial^1} C^2 _L (M, \mathcal{A}) \xrightarrow{\partial^2} ...
\end{equation}
Then, we define as usual (see Section 1.1):
\begin{equation}\label{Eq:1.4}
H^n _L (M, \mathcal{A}):= H^n (C^{\bullet} _L (M, \mathcal{A}))= \ker \partial^{n+1} / \Ima \partial^n
\end{equation}
By [14] we know that:
\begin{equation}\label{Eq:1.5}
H^n _L (M, \mathcal{A}) = H^n (C^{\bullet} _L (M, \mathcal{A})) = H^n (\mathcal{D}M, \mathcal{A})
\end{equation}
To construct our new general cohomology theory we will need to use the abelian groups $C^p _L (M, \mathcal{A})$ and the coboundary maps $\partial^p$, as shown in the next section.
%%%%%%%%%%%%%%%%%%%%%
\section{A general cohomology theory for a sequence of monoids (square and rectangular cohomologies)}
We now develop a general cohomology theory starting from a sequence of monoids. The main idea is to construct an infinite square or rectangle using the cohomology theories that we associate to each element of the sequence, and then to follow a path $\pi$ in one of these figures in order to select some abelian groups and some homomorphisms that will form a new cochain. Finally, from this cochain we will obtain a new cohomology theory, this time associated to the sequence. It will be clear that the new cohomology theory usually consist of many cohomology groups which are exactly the same as some cohomology groups associated to some monoids. This is one of the aims: our new cohomology theory should encode as much information related to the various monoids as possible, not necessarily by using new cohomology groups (which will actually be the case, generally), but also using the ones of the other cohomology theories together, in a unique (new) cohomology. In fact, \textit{the most interesting situation is when we have something to which we associate a sequence and from which we then construct a cohomology theory as explained above. We can then forget about the sequence, and use the cohomology theory for that 'something'.} This is exactly what we will do in Section 3 and 4.\\
Consider $\lbrace M_n \rbrace_{n=0}^{\infty}$, where all the $M_n$'s are monoids (we assume that $M_j \neq M_i$ $\forall i \neq j$; later we will consider the finite case, which will be used in particular in Section 4). The idea is to use Leech cohomology for all these monoids and then construct another cohomology theory using a selection process that involves a chosen path $\pi$. The reason why we do this is suggested by our particular examples; see Section 3 below for a complete motivation of this method.\\
We want to associate to each monoid a certain abelian group; since it is not always possible to use Grothendieck completion (meaning that sometimes it leads to trivial results, see Section 3) and it is not always possible to add elements to a monoid in order to make into a group (again, see Section 3 for an example), we need to find another way to do this. Leech cohomology will solve this problem. Let $_n C^p _L (M_n,\mathcal{A}_n)$ be the abelian groups in the cochain associated to $M_n$ (with coefficient in $\mathcal{A}_n$), and let $_n \partial^p$ be the corresponding coboundary maps. We have the following:
\begin{multline}\label{Eq:2.1}
0 \rightarrow \, _0 C^0 _L (M_0, \mathcal{A}_0) \xrightarrow{ _0 \partial^0} \, _0 C^1 _L (M_0, \mathcal{A}_0) \xrightarrow{ _0 \partial^1} \, _0 C^2 _L (M_0, \mathcal{A}_0) \xrightarrow{ _0 \partial^2} ...\\
0 \rightarrow \, _1 C^0 _L (M_1, \mathcal{A}_1) \xrightarrow{_1 \partial^0} \, _1 C^1 _L (M_1, \mathcal{A}_1) \xrightarrow{_1 \partial^1} \, _1 C^2 _L (M_1, \mathcal{A}_1) \xrightarrow{_1 \partial^2} ... \qquad \qquad \quad \\
0 \rightarrow \, _2 C^0 _L (M_2, \mathcal{A}_2) \xrightarrow{_2 \partial^0} \, _2 C^1 _L (M_2, \mathcal{A}_2) \xrightarrow{_2 \partial^1} \, _2 C^2 _L (M_2, \mathcal{A}_2) \xrightarrow{_2 \partial^2} ... \qquad \qquad \quad
\end{multline}
If we delete all the trivial groups on the left, we can see the remaing part as a square where each side has infinite lenght (the new cohomology theory we are developing now will be then called 'square cohomology'). We want to form a cochain starting from these abelian groups: in order to do this, we need a bit of preparation. First of all, we also want to have other cochains; in particular, we want to define \footnote{Actually, we will consider any family of such homomorphisms, without fixing one of them for every case; this is because we prefer to be really general (indeed, as we will see below, some families give rise to double cochain complexes, while others do not; we will thus have more possibilities to choose the most appropriate and useful cohomology for the particular situation by maintaining this general fashion with the vertical coboundary maps). At least one such family always exists (the trivial one), as we will see later.} some coboundary maps $_n \widetilde{\partial}^p : \, _n C^p _L (M_n, \mathcal{A}_n) \rightarrow \, _{n+1} C^p _L (M_{n+1}, \mathcal{A}_{n+1})$ so that the following cochains give rise to other cohomology theories for each $n \geq 0$ (the cohomology groups are defined as usual; since we will not need them here, we will not explicitely write them):
\begin{equation}\label{Eq:2.2}
0 \rightarrow \, _0 C^n _L (M_0, \mathcal{A}_0) \xrightarrow{ _0 \widetilde{\partial}^0} \, _1 C^n _L (M_1, \mathcal{A}_1) \xrightarrow{ _0 \widetilde{\partial}^1} \, _2 C^n _L (M_2, \mathcal{A}_2) \xrightarrow{ _0 \widetilde{\partial}^2} ...
\end{equation}
We can write the following (which is not necessarily a double complex, because we will assume, more generally, \eqref{Eq:2.3}; see Section 2.1):
\begin{align*}
0 \rightarrow \, _0 C^0 _L (M_0, \mathcal{A}_0) \xrightarrow{ _0 \partial^0} \, _0 C^1 _L (M_0, \mathcal{A}_0) \xrightarrow{ _0 \partial^1} \, _0 C^2 _L (M_0, \mathcal{A}_0) \xrightarrow{ _0 \partial^2} ...\\
\downarrow \, _0 \widetilde{\partial}^0 \qquad \qquad \qquad \downarrow \, _0 \widetilde{\partial}^1 \qquad \qquad \qquad \downarrow \,_0 \widetilde{\partial}^2 \qquad \qquad\\
0 \rightarrow \, _1 C^0 _L (M_1, \mathcal{A}_1) \xrightarrow{_1 \partial^0} \, _1 C^1 _L (M_1, \mathcal{A}_1) \xrightarrow{_1 \partial^1} \, _1 C^2 _L (M_1, \mathcal{A}_1) \xrightarrow{_1 \partial^2} ...\\
\downarrow \, _1 \widetilde{\partial}^0 \qquad \qquad \qquad \downarrow \, _1 \widetilde{\partial}^1 \qquad \qquad \qquad \downarrow \,_1 \widetilde{\partial}^2 \qquad \qquad\\
0 \rightarrow \, _2 C^0 _L (M_2, \mathcal{A}_2) \xrightarrow{_2 \partial^0} \, _2 C^1 _L (M_2, \mathcal{A}_2) \xrightarrow{_2 \partial^1} \, _2 C^2 _L (M_2, \mathcal{A}_2) \xrightarrow{_2 \partial^2} ...\\
\downarrow \, _2 \widetilde{\partial}^0 \qquad \qquad \qquad \downarrow \, _2 \widetilde{\partial}^1 \qquad \qquad \qquad \downarrow \,_2 \widetilde{\partial}^2 \qquad \qquad \\
\, ... \qquad \qquad \qquad \qquad  \, ... \qquad \qquad \qquad \qquad \, ... \qquad \qquad \quad \,
\end{align*}
This sums up everything. We now define the notion of path 'connected by couples'. We will briefly indicate each $_n C^p _L (M_n, \mathcal{A}_n)$ by the couple $(n,p) \in \mathbb{N}^2 _0$ (the order matters). It is clear that there is a bijection between the set of all the abelian groups in the square above and all the couples in $\mathbb{N}^2 _0$. Consider a map $\pi_B \equiv \pi :B \subseteq \mathbb{N}^2 _0 \rightarrow B \setminus \lbrace(0,0) \rbrace$ which satisfies the following conditions (we require them for reasons that will be explained below):\\
(i) $ (0,0) \in B$;\\
(ii) $B=\Dom \pi$ is such that, whenever $(i,j) \in B$, one and only one element of $\lbrace (i+1,j), (i,j+1) \rbrace$ belongs to $B$;\\
(iii)
$$ \pi(i,j):= \begin{cases} (i+1,j), & \text{if}\ (i+1,j) \in B \\ (i,j+1), & \text{if}\ (i,j+1) \in B \  \end{cases}$$
Note that $B=\Ima \pi \, \cup \, \lbrace(0,0) \rbrace =\Dom \pi$. Such a map $\pi$ is called a path 'connected by couples'. The reason of (i) is that we want to start with the first abelian group given by Leech cohomology related to $M_0$ (this way, we do not start "too far", loosing some groups; note that, by this definition,
$$\lbrace i \in \mathbb{N}_0 : (i,\widehat{i})\in B \, \text{for some} \, \widehat{i} \in \mathbb{N}_0, \, j \in \mathbb{N}_0 : (\widehat{j},j)\in B \, \text{for some} \, \widehat{j} \in \mathbb{N}_0 \rbrace = \mathbb{N}_0$$
which means that we cover all $\mathbb{N}_0$ with these $i,j$'s). This formalise the idea of "not going too far"; indeed, we want to cover the whole $\mathbb{N}_0$). The reasons of (ii) and (iii) give rise to the name 'connectedness by couples', and can be explained with the following example (where we do not write the arrows which do not show the path):
\begin{align*}
0 \qquad \, \underline{_0 C^0 _L (M_0, \mathcal{A}_0)} \qquad \, _0 C^1 _L (M_0, \mathcal{A}_0) \qquad \, _0 C^2 _L (M_0, \mathcal{A}_0) \qquad ...\\
\downarrow \, _0 \widetilde{\partial}^0 \qquad \qquad \qquad \qquad \qquad \qquad \qquad \qquad \qquad \qquad\\
0 \qquad \, \underline{_1 C^0 _L (M_1, \mathcal{A}_1)} \xrightarrow{_1 \partial^0} \, \underline{_1 C^1 _L (M_1, \mathcal{A}_1)} \qquad \, _1 C^2 _L (M_1, \mathcal{A}_1) \qquad ...\\
\qquad \qquad \qquad \qquad \qquad \qquad \downarrow \, _1 \widetilde{\partial}^1 \qquad \qquad \qquad \qquad \qquad \qquad\\
0 \qquad \, _2 C^0 _L (M_2, \mathcal{A}_2) \qquad \, \underline{_2 C^1 _L (M_2, \mathcal{A}_2)} \xrightarrow{_2 \partial^1} \, \underline{_2 C^2 _L (M_2, \mathcal{A}_2)} \qquad ...
\end{align*}
Note that the underlined groups, which can be represented by couples as said above, are connected by the arrows of the coboundary maps $_n \partial^p$ and $_n \widetilde{\partial}^p$. The first values of the map $\pi$, in this case, are given by: $\pi(0,0)=(1,0)$, $\pi(1,0)=(1,1)$, $\pi(1,1)=(2,1)$, $\pi(2,1)=(2,2)$. It is clear that the map $\pi$ is connected by the couples through the arrows of the homomorphisms; from this we take the name (and also (ii) and (iii)). Note that we also consider $(0,0)$ among the couples (it belongs to $\Dom \pi = \Ima \pi \, \cup \, \lbrace (0,0) \rbrace$, and it is important to connect it to the path given by $\pi$). Now consider any path $\pi_B$ connected by couples. We will form a cochain using the groups corresponding to the couples in $\Ima \pi \, \cup \lbrace (0,0) \rbrace$. For instance, using the path $\pi$ of the previous example, we have:
$$ 0 \rightarrow \, _0 C^0 _L (M_0, \mathcal{A}_0) \xrightarrow{_0 \widetilde{\partial}^0} \, _1 C^0 _L (M_1, \mathcal{A}_1) \xrightarrow{_1 \partial^0} \, _1 C^1 _L (M_1, \mathcal{A}_1) \xrightarrow{_1 \widetilde{\partial}^1} \, _2 C^1 _L (M_2, \mathcal{A}_2) \xrightarrow{_2 \partial^1} ...$$
For this to be a cochain, we need to require (in general) that:
\begin{multline}\label{Eq:2.3}
\widetilde{\partial}(\pi(i,j)) \circ \partial(i,j) = 0 \quad \forall (i,j) \in B \, | \, \pi_1( \pi(i,j))=i+1\\
\partial(\pi(i,j)) \circ \widetilde{\partial}(i,j) = 0 \quad \forall (i,j) \in B \, | \, \pi_2( \pi(i,j))=j+1 \qquad \qquad \qquad \quad \, \, \, \,
\end{multline}
where here, for simplicity of notation, we have written $\partial(n,p):= \, _n \partial^p$ and $ \widetilde{\partial}(n,p):= \, _n \widetilde{\partial}^p$ (the order matters), and where $\pi_1$ and $\pi_2$ are the projections w.r.t. the first, second component, respectively (i.e. $\pi_1(a,b)=a$, $\pi_2(a,b)=b$). The identity $\partial(n,p+1) \circ \partial(n,p)=0$ is always satisfied, so there was no need to write it again; moreover, since we have consider cohomologies for the cochains $_{\bullet} C^n _L (M_{\bullet}, \mathcal{A}_{\bullet})$ ($n$ fixed), we have already required also that $\widetilde{\partial}(n+1,p) \circ \widetilde{\partial}(n,p)=0$ must hold $\forall n$ and for each fixed $p$; thus, even in this case we have not written it again. Notice that the cochains constructured using $\pi$ also depend on the chosen family $\lbrace \, _n \widetilde{\partial}^p \rbrace_{n,p \in \mathbb{N}_0}$ of homomorphisms (which are required to satisfy all the above equations). We also note that there always exists at least one family of homomorphisms satisfying all the above identities, that is, the one given by $_n \widetilde{\partial}^p \equiv 0$ $\forall n,p$ (of course, whenever possible, we prefer to consider non trivial families of homomorphisms). \textit{Even in the cases where we have the \emph{trivial family}} (which occur quite often), we still have some interesting cohomology theory. More precisely, we will only have two new kinds of cohomology groups, which will be called "trivial" since they come from the trivial vertical homomorphisms (actually, they are usually different from the "true" trivial group, namely $(0)$). These groups do not cause problems, because they are only needed  to go from a "floor" to another "floor" of the square or rectangle. Moreover, it is also possible (see below) to "avoid in the distance" these groups, so that their presence does not negatively affect the theory. Thus, we can always consider the trivial vertical homomorphisms, and in fact we will use them whenever needed (that is, whenever we want a simple and concrete example of family of vertical homomorphisms).\\
We now show some examples where we "avoid in the distance" the trivial homomorphisms. For instance, we could let $\pi(i,j):=(i+1,0)$ $\forall i$ (i.e. we have the same cohomology theory of Leech cohomology for $M_0$), which would avoid all the trivial homomorphisms, or, using some trivial homomorphisms (which lead to some $_n C^p _L$, $\ker \partial(i,j)$ as cohomology groups, see Proposition \ref{Prop:2.1} below; these are the groups which we called "trivial"), we could have something more interesting, like: 
$$\pi(i,j)=\begin{cases} (i,j+1), & \text{if \textit{i} is not prime}\  \\ (i+1,j), & \text{if \textit{i} is prime}\  \  \end{cases}$$
for which some cohomology groups are indeed "trivial", but they get farther and farther from each other when $i$ gets larger and larger. Another example, which is similar to the previous one but which avoids the "large" number of primes before, say, $30$ (we could choose any number, of course), is given by:
$$\pi(i,j)=\begin{cases} (i,j+1), & \text{if \textit{i} is not prime or if it is a prime}\ \leq 30   \\ (i+1,j), & \text{if \textit{i} is prime}\ >30 \  \end{cases}$$
Yet another example: we could also define $\pi$ so that it always goes on the right (i.e. $\pi(i,j):=(i,j+1)$) except for a finite number of times (i.e. $\pi(i,j):=(i+1,j)$ only for a finite number of $(i,j)$'s). This method to "avoid in the distance" the trivial groups is actually also useful when discussing local Eilenberg-Steenrod axioms, see Section 2.2.\\
We now prove the result announced above:
\begin{proposition}\label{Prop:2.1}
The "trivial groups" which are obtained from the trivial family of vertical homomorphisms are $_{i+1} C^j _L$ and $\ker \partial(i,j+1)$.
\end{proposition}
\begin{proof}
If $\widetilde{\partial}(i,j) \equiv 0$, then we could have two cases (or more if we generalise the notion of path connected by couples; see Remark \ref{Rm:2.1}): the next homomorphism is either $\widetilde{\partial}(\pi(i,j))$ (which is thus equal to $\widetilde{\partial}(i+1,j)$) or $\partial(\pi(i,j))$ (which is thus equal to $\partial(i,j+1)$). In the former case, $\ker \widetilde{\partial}(\pi(i,j)) / \Ima \widetilde{\partial}(i,j) \equiv \, _{i+1} C^j _L / \lbrace 0 \rbrace$; in the latter case, $\ker \partial(\pi(i,j)) / \Ima \widetilde{\partial}(i,j) \equiv \ker \partial(i,j+1) / \lbrace 0 \rbrace$. It is well known that, if $G$ is a group, then $G \, / \, \lbrace 0 \rbrace$ is isomorphic to $G$ (for instance, consider the identity map and use the first Isomorphism Theorem). Thus, we obtain the cohomology groups $_{i+1} C^j _L$, $\ker \partial(i,j+1)$, respectively.
\end{proof}
We conclude by defining explicitely the cohomology groups. Here, with $d^p$ we will indicate either $\partial(n,p)$ or $\widetilde{\partial}(n,p)$, depending on the chosen $\pi_B$:
\begin{equation}\label{Eq:2.4}
H^p _S(\lbrace M_n \rbrace_{n \in \mathbb{N}_0}):=H^p(\lbrace M_n \rbrace_{n \in \mathbb{N}_0}, \lbrace \mathcal{A}_n \rbrace_{n \in \mathbb{N}_0}, \pi_B, \lbrace \widetilde{\partial}(p,j) \rbrace_{p,j \in \mathbb{N}_0})= \ker d^{p+1} / \Ima d^p
\end{equation}
This cohomology theory is called 'square cohomology' since we used the 'infinite square' $\mathbb{N}^2 _0$ (the $S$ in $H^p _S$ clearly comes from this name). We can sum up everything in the following important Theorem:
\begin{theorem}[Square Cohomology - Infinite Case] \label{Thm:2.1}
Let $\lbrace M_n \rbrace_{n \in \mathbb{N}_0}$ be a sequence of monoids such that $M_i \neq M_j$ $\forall i \neq j$. Let $_n C^{\bullet} _L (M_n, \mathcal{A}_n)$ denote the cochain given by Leech cohomology for each $M_n$ with coefficients in $\mathcal{A}_n$. Consider some path $\pi_B : B \subseteq \mathbb{N}^2 _0 \rightarrow B \setminus \lbrace (0,0) \rbrace$ that is connected by couples. Let $_n \widetilde{\partial}^p : \, _n C^p _L (M_n, \mathcal{A}_n) \rightarrow \, _{n+1} C^p _L (M_{n+1}, \mathcal{A}_{n+1})$ be homomorphisms such that $_{n+1} \widetilde{\partial}^p \circ \, _n \widetilde{\partial}^p = 0$ $\forall p$ and for each fixed $n \geq 0$. Suppose that this family of homomorphisms and the homomorphisms given by Leech cohomology satisfy \eqref{Eq:2.3} (at least the trivial family $_n \widetilde{\partial}^p \equiv 0$ satisfy all these hypothesis). Then, we can form a cochain consisting of $_0 C^0 _L(M_0, \mathcal{A}_0)$ and all the abelian groups given by the corresponding couples of $\pi_B$ (where we follow the order of the path described by $\pi_B$, as shown above) and we then have a cohomology theory, called square cohomology, associated to the sequence $\lbrace M_n \rbrace_{n \in \mathbb{N}_0}$. The cohomology groups $H^p _S(\lbrace M_n \rbrace_{n \in \mathbb{N}_0})$ are defined by \eqref{Eq:2.4}.
\end{theorem}
We now give a finite version of this result. It is clear that the previous discussion still holds with $\lbrace M_n \rbrace_{n=0}^{k}$, $k \neq \infty$. The only problem to fix is that, this time, if we reach some group on the last line (i.e. if we reach some $_k C^n _L$), we cannot obviously go beyond it (because we do not have a monoid $M_{k+1}$ that gives us the cochain of its Leech cohomology). Thus, we need to slightly change the definition of $\pi_B$ connected by couples, because we cannot have $\pi(k,j)=(k+1,j)$, so add the condition that $\pi(k,j):=(k,j+1)$ $\forall j$. Everything else can be easily reduced to the finite case, obtaining the following (since this time we have $\lbrace 0, 1, ..., k \rbrace \times \mathbb{N}_0$, this will be called rectangular cohomology):
\begin{theorem}[Rectangular Cohomology - Finite Case]\label{Thm:2.2}
Let $\lbrace M_n \rbrace_{n=0}^{k}$ consist of monoids such that $M_i \neq M_j$ $\forall i \neq j$. Let $_n C^{\bullet} _L (M_n, \mathcal{A}_n)$ denote the cochain given by Leech cohomology for each $M_n$ with coefficients in $\mathcal{A}_n$. Consider some path $\pi_B : B \subseteq \mathbb{N}^2 _0 \rightarrow B \setminus \lbrace (0,0) \rbrace$ that is connected by couples. Let $_n \widetilde{\partial}^p : \, _n C^p _L (M_n, \mathcal{A}_n) \rightarrow \, _{n+1} C^p _L (M_{n+1}, \mathcal{A}_{n+1})$ be homomorphisms such that $_{n+1} \widetilde{\partial}^p \circ \, _n \widetilde{\partial}^p = 0$ $\forall p$ and for each fixed $0 \leq n \leq k$. Suppose that this family of homomorphisms and the homomorphisms given by Leech cohomology satisfy \eqref{Eq:2.3} (at least the trivial family $_n \widetilde{\partial}^p \equiv 0$ satisfy all these hypothesis). Then, we can form a cochain consisting of $_0 C^0 _L(M_0, \mathcal{A}_0)$ and all the abelian groups given by the corresponding couples of $\pi_B$ (where we follow the order of the path described by $\pi_B$, as shown above) and we then have a cohomology theory, called rectangular cohomology, associated to $\lbrace M_n \rbrace_{n=0}^{k}$. The cohomology groups $H^p _R(\lbrace M_n \rbrace_{n=0}^{k})$ are defined by \eqref{Eq:2.4} (change $S$ with $R$ and $n \in \mathbb{N}_0$ with $n=0, 1, ..., k$).
\end{theorem}
As we wanted, in these cohomologies we mainly have cohomology groups which are the same as some groups of the cohomologies on the various floors. In addition, we only have the "trivial groups" which allow us to go from a floor to another one. Thus, we have succesfully encoded as much information as possible (choosing a proper path).
\begin{remark}\label{Rm:2.1}
\normalfont We also note that we defined $\pi_B$ in such a way that it is "weakly decreasing", meaning that it always goes either on the same cochain or it goes on the cochain below it. Of course, the situation can be easily generalised so that the path, still assumed to be connected by couples, is allowed to go to a cochain above (whenever possible; when we reach some group $_0 C^n _L$, we cannot clearly go above it). To do this, we will also need a refinement of equation \ref{Eq:2.3}, which is simple to obtain, following the reasoning preceding Theorem \ref{Thm:2.1}.\\
Similarly, we can also think of going to the left, or going diagonal, ... All of these requirements can be easily added with some minor and simple modifications, as in the case of 'going up', described above. We conclude saying that the method of the square (or rectangular) cohomology can obviously be applied also in other situations where we have some sequence and we can find some (also different from each other) cohomology theories for each element of the sequence (for instance, to the first term we can associate sheaf cohomology, to the second term we can associate de Rham cohomology, ...). The same applies for homology theories, of course. Another possible generalisation includes the possibility of associating, to some elements, cohomology theories, and associating homology theories to other elements (clearly, we will have the horizontal arrows which go on the right for cohomologies, and on the left for homologies (or viceversa; see Section 1.1)). It is also because of all these possible generalisations that we prefer to maintain a certain level of generality in the choice of the vertical homomorphisms. Note that the trivial ones, which as already said always work properly, allow us to define these generalisations without difficulty.
\end{remark}
The infinite case will be applied to $f_s$, while the finite case will be applied to both $f_s$ and $h$ in the next sections. In Section 2.2 we will discuss Eilenberg-Steenrod axioms for square and rectangular (co)homology.
\subsection{Double complexes and square/rectangular cohomologies}
The construction of the square (or, in the finite case, of the rectangle) is really similar to a double complex (for more details on double complexes, see [36-39]). However, our cohomology theories are more general, thanks to \eqref{Eq:2.3}. If our construction turns out to be a double complex \footnote{For example, if we consider the trivial family of homomorphisms $_i \widetilde{\partial}^j$, we always have a double complex.} (which is not necessarily the case, because we do not require (and \eqref{Eq:2.3} does not imply) that every square in the diagram commutes), we can construct cohomology groups $H^{p,q}$ in a similar way to $H^p$. This cohomology theory clearly encodes everything about the groups (and not 'as much as possible', as for square/rectangular cohomology) and it is a valid alternative to square/rectangular cohomology. Since in this paper we are mainly dealing with cochain complexes (not double cochain complexes), we would also like to have a cochain with a corresponding cohomology theory derived from the double cochain above. This can be done by using \footnote{There are also other ways to obtain cochain complexes from double complexes; see, for instance, [42].} the total complex (w.r.t. sum, in this case):
\begin{equation}\label{Eq:2.5}
\Tot(C)_n:=\bigoplus_{p+q=n} C^{p,q}
\end{equation}
where $\Tot(C)_n$ is the $n$-th (excluding $(0)$) term in the new cochain complex, and $C^{\bullet,\bullet}:=(C^{p,q})_{p,q \in \mathbb{N}_0}$ is the double complex. The coboundary maps are defined as a certain linear combination of the horizontal and vertical coboundary maps. We refer to Section 1.2 of [52] for more details. We also note that a useful method to compute the cohomology groups of the total complex is to use a certain spectral sequence, namely the spectral sequence of a double complex; see Chapter III in [25], [38] and [40-41] for more on these topics.\\
In such cases, we can choose the most useful (which depends on what we want to do with it) cohomology theory for a sequence of objects $\lbrace A_j \rbrace$. Hence, in some particular situations we can also consider other cohomology theories for structured spaces (see Sections 3 and 4).\\
We conclude noting that, if the square/rectangle is actually a double complex, the salamander lemma (see [39]) uses a similar idea to our "following the path $\pi$". However, this lemma gives some exact sequences, and it is applicable with double complexes, while our method of the path does not require that squares commute and does not generate, usually, an exact sequence, but a cochain complex. Thus, our method works in more general situations, even though in some cases (and when it is possible) it could be more useful to consider a double complex and the associated (via \eqref{Eq:2.5}) total complex instead of square/rectangular cohomology. Summarising, we have followed these steps:\\
1) We wanted a cohomology theory associated to a sequence, so we started choosing a cohomology theory for each element of the sequence (Leech cohomology, which will be used in this paper, but also other (co)homologies, as explained in Remark \ref{Rm:2.1});\\
2) We then had two possibilities: the square/rectangle obtained could either be a double cochain complex or not. In the latter case, we use square/rectangular cohomology, where we follow a path $\pi$, while in the former we can choose among square/rectangular cohomology, the cohomology associated to the double complex, and the cohomology given by the total cochain. This choice is made depending on what we want to do with the cohomology theory.
%%%%%%%%%%%%%%%%%%%%%%%%%
\subsection{Local Eilenberg-Steenrod axioms}
In this subsection we show that Eilenberg-Steenrod axioms locally hold for square and rectangular (co)homologies (recall that homologies can be obtained as described in Remark \ref{Rm:2.1}). We start recalling the five ES axioms for homology (see also [46-50]):\\
\\
(Homotopy) Homotopic maps induce the same map in homology;\\
(Excision) If $(X,A)$ is a pair and $U$ is an open subset of $X$ whose closure is contained in the interior of $A$, the inclusion map $i:(X \setminus U, A \setminus U) \rightarrow (X,A)$ induces an isomorphism in homology;\\
(Exactness) Each pair $(X,A)$ induces a long exact sequence in homology:
$$ ... \rightarrow H_j(A) \xrightarrow{i_*} H_j(X) \xrightarrow{t_*} H_j(X,A) \xrightarrow{\partial} H_{j-1}(A) \rightarrow ... $$
where $i:A \rightarrow X$ and $t: X \rightarrow (X,A)$ are the inclusion maps;\\
(Additivity) The homology of a disjoint union is the direct sum of the homologies, that is:
$$ H_j( \coprod_i X_i) \cong \bigoplus_i H_j(X_i) $$
(Dimension) $H_n(\lbrace \pt \rbrace)=0$ $\forall n \neq 0$, where $\lbrace \pt \rbrace$ is any singleton set\\
\\
where the $H_n$'s are functors from the category of pairs $(X, A)$ of topological spaces to \textbf{Ab}, together with a natural transformation $\partial : H_j(X, A) \rightarrow H_{j-1}(X, \emptyset)=:H_j(X)$. Via the usual "dualisation" process, we get the axioms for cohomology. If some of the above axioms do not hold, we obtain more general (co)homology theories. For example, if the dimension axiom is not satisfied (as for K-theories, ...), we have a generalised (co)homology theory (see, for instance, [48]). Moreover, some (co)homology theories satisfy some "slightly different" ES axioms (an example is given by [51]). All these possibilities will still hold, "locally", in the square/rectangular (co)homology constructed starting from them, as we will now show.\\
Roughly speaking, when we construct a square (or rectangle) which gives us square (or rectangular) (co)homology, if the (co)homology theory at some line $n$ satisfies the ES axioms (or it satisfies some of them, or some slightly different axioms), then, if the chosen path $\pi$ includes enough groups of this theory with continuity (i.e. without changing the "floor"), then the (co)homology groups will be the same for both the theory itself and for the square (or rectangular) (co)homology. Thus, those groups will still satisfy the axioms. This can be generalised and formalised, as explained below.\\
Given a square (or rectangle) of some square (or rectangular) (co)homology, a local (co)chain is any continuous sequence (where we follow the directions of the arrows) of abelian groups belonging to the chosen path $\pi$, where the term 'continuous' means that 'every groups in the (co)chain is connected to other ones via a horizontal or vertical homomorphism'. For instance, if this is the square of the square cohomology with Leech groups as in Section 2:
\begin{align*}
0 \rightarrow \, _0 C^0 _L (M_0, \mathcal{A}_0) \xrightarrow{ _0 \partial^0} \, _0 C^1 _L (M_0, \mathcal{A}_0) \xrightarrow{ _0 \partial^1} \, _0 C^2 _L (M_0, \mathcal{A}_0) \xrightarrow{ _0 \partial^2} ...\\
\downarrow \, _0 \widetilde{\partial}^0 \qquad \qquad \qquad \downarrow \, _0 \widetilde{\partial}^1 \qquad \qquad \qquad \downarrow \,_0 \widetilde{\partial}^2 \qquad \qquad\\
0 \rightarrow \, _1 C^0 _L (M_1, \mathcal{A}_1) \xrightarrow{_1 \partial^0} \, _1 C^1 _L (M_1, \mathcal{A}_1) \xrightarrow{_1 \partial^1} \, _1 C^2 _L (M_1, \mathcal{A}_1) \xrightarrow{_1 \partial^2} ...\\
\downarrow \, _1 \widetilde{\partial}^0 \qquad \qquad \qquad \downarrow \, _1 \widetilde{\partial}^1 \qquad \qquad \qquad \downarrow \,_1 \widetilde{\partial}^2 \qquad \qquad\\
0 \rightarrow \, _2 C^0 _L (M_2, \mathcal{A}_2) \xrightarrow{_2 \partial^0} \, _2 C^1 _L (M_2, \mathcal{A}_2) \xrightarrow{_2 \partial^1} \, _2 C^2 _L (M_2, \mathcal{A}_2) \xrightarrow{_2 \partial^2} ...\\
\downarrow \, _2 \widetilde{\partial}^0 \qquad \qquad \qquad \downarrow \, _2 \widetilde{\partial}^1 \qquad \qquad \qquad \downarrow \,_2 \widetilde{\partial}^2 \qquad \qquad \\
\, ... \qquad \qquad \qquad \qquad  \, ... \qquad \qquad \qquad \qquad \, ... \qquad \qquad \quad \,
\end{align*}
then the following one \footnote{In all these examples, we tacitly assume that the chosen path $\pi$ includes the considered groups in its image.} is a local cochain:
$$ _1 C^0 _L (M_1, \mathcal{A}_1) \xrightarrow{_1 \widetilde{\partial}^0} \, _2 C^0 _L (M_2, \mathcal{A}_2) \xrightarrow{_2 \partial^0} \, _2 C^1 _L (M_2, \mathcal{A}_2) $$
while
$$ _1 C^0 _L (M_1, \mathcal{A}_1) \rightarrow \, _2 C^1 _L (M_2, \mathcal{A}_2) $$
is not a local cochain, because there is no horizontal or vertical homomorphism connecting them (composition of vertical and/or horizontal morphisms is not taken into account in this definition). Of course, it is possible to add some elements and obtain a local cochain (for instance, the previous one can be seen as a "completion" of this one). A horizontal (or vertical) local (co)chain is a local (co)chain whose elements are on the same "floor" (or on the same column). For instance, the following one is a horizontal local cochain:
$$_2 C^0 _L (M_2, \mathcal{A}_2) \xrightarrow{_2 \partial^0} \, _2 C^1 _L (M_2, \mathcal{A}_2) \xrightarrow{_2 \partial^1} \, _2 C^2 _L (M_2, \mathcal{A}_2)$$
and this one is a vertical local cochain:
$$ _0 C^0 _L (M_0, \mathcal{A}_0) \xrightarrow{_0 \widetilde{\partial}^0} \, _1 C^0 _L (M_1, \mathcal{A}_1) $$
The initial element of a local (co)chain is the first element of such sequence. Similarly, the final element of a local (co)chain is the last element of the sequence. For instance, in the example of the local horizontal cochain above, $_2 C^0 _L (M_2, \mathcal{A}_2)$ is the initial element, while $_2 C^2 _L (M_2, \mathcal{A}_2)$ is the final element. After fixing some path $\pi$, we say that a (co)homology group is an extremal group if it is $\cong \ker d_p / \Ima d_{p+1}$ ($\cong \ker d^{p+1} / \Ima d^{p}$, respectively) for $d_p$ horizontal/vertical and $d_{p+1}$ vertical/horizontal homomorphism (similarly for cohomology) \footnote{Of course, the maps $d_p$ ($d^p$) are the ones obtained following the fixed path $\pi$, that is, they are the ones of the square (or rectangular) (co)homology w.r.t. $\pi$.}. This implies that non-extremal groups are the same as some (co)homology groups of the (co)homology theory on a certain floor. These notions allow us to formalise the meaning of "local axioms" \footnote{The term 'local' comes from these notions.}. Actually, this will be only needed for exactness; the other axioms are easier to treat. We now prove the following important result, which states what happens in general when we consider square or rectangular (co)homologies with (co)homology theories satisfying the ES axioms.
\begin{theorem}[Local ES axioms for square/rectangular (co)homology]\label{Thm:2.3}
Consider a square or rectangular (co)homology (w.r.t. some fixed path $\pi$) where all the (co)homology theories on each floor satisfy the ES axioms. Then, the square or rectangular (co)homology satisfies the 'local Eilenberg-Steenrod axioms', that is, the following local version of the ES axioms:\\
(Homotopy) If $H_i$ ($H^i$) is a non-extremal (co)homology square (or rectangular) group, then: homotopic maps induce the same map in such a (co)homology \footnote{That is, the induced maps are certainly the same when we consider non-extremal (co)homology groups.}; more briefly, we say that Homotopy holds locally;\\
(Excision) Excision holds locally, that is, with non-extremal groups;\\
(Additivity) Additivity holds locally, that is, with non-extremal groups;\\
(Dimension) The dimension axiom holds locally, that is, with non-extremal groups;\\
(Exactness) Every local horizontal (co)chain induces an exact sequence of (co)homology groups, where the sequence is considered without the extremal (co)homology groups (which are in fact deleted from it).
\end{theorem}
\begin{proof}
We noticed above, when defining the extremal groups, that the non-extremal groups are in fact the same as the groups of some (co)homology theory on some floor. Thus, it is clear that the four axioms above hold locally, with the meaning that indeed they certainly hold w.r.t. non-extremal (co)homology groups. Local exactness is obtained in a similar way, noting that in this case we need to delete the extremal groups from the sequence to assure exactness.
\end{proof}
\begin{remark}\label{Rm:2.2}
\normalfont If the path $\pi$ is chosen in a proper way, the exactness axiom will give us sequences with many terms. For instance, in order to do this we may choose $\pi$ so that the vertical local chains are short (at most two or three terms) while the horizontal local cochains are long (at least five terms, for example). Clearly, this last part always happens, at some point, when we deal with rectangular (co)homology in its "basic" form (i.e. not the generalisations considered in Remark \ref{Rm:2.1}). Indeed, at a certain floor we have to stop: we cannot go below it anymore. Thus, at that point we will certainly have an infinite (or finite, if the (co)chain is bounded) local horizontal (co)chain, which will give a long exact sequence that is "similar to the usual ones" given by the global exactness axiom. Also note that the method of "avoiding in the distance" the trivial groups may be useful to assure that these axioms locally hold.
\end{remark}
It is possible, in some cases (depending on the chosen vertical homomorphisms), that the ES axioms also hold in the usual sense, i.e. globally. However, the locality of the axioms above already works properly. Indeed, even though we may loose information in some of the extremal cases, most of it is maintained. Actually, we can only loose "passage" information, that is, information that encodes the passage from a (co)homology theory on a floor to the (co)homology theory on the floor below it (or above it; remember Remark \ref{Rm:2.1}).\\
We conclude noting that, actually, the above result can be generalised in various directions.
\begin{remark}\label{Rm:2.3}
\normalfont A first generalisation can be obtained considering less axioms. For instance, if all the (co)homology theories on each floor are actually generalised (co)homology theories (i.e. they do not satisfy the dimension axiom), then Theorem \ref{Thm:2.3} still holds (without the local dimension axiom, clearly). Similarly, if instead of the ES axioms we have something slightly different (an example can be found in [51]), the same axioms will continue to hold, in a local version, also for square and rectangular (co)homology. Furthermore, an even more general case is admissible: if only some theories (on certain floors) satisfy some kinds of axioms (even different from each other), these will still hold, locally, for square and rectangular (co)homology. Clearly, even in such cases, as before, we will exclude the extremal groups (for which the axioms may or may not hold, depending on the chosen vertical homomorphisms).
\end{remark}
%%%%%%%%%%%%%%%%%%%%%%%%%%%
\section{A cohomology theory related to $f_s$}
We now start to develop a cohomology theory which arises from $f_s$ (see Section 2 in [1]). We will use $\cup$, $\cap$ and $\subseteq$ with collections of sets in the same way as for sets (as we did in [1]). In the cited paper we used the symbol $\mathcal{T}$ to indicate the space of all the local structures $f_s(U_p)$ of a structured space $X$. We will adopt a similar notation in what follows. First of all, start with an algebraic structure $X=:X_0$. Let $A_1$ be an algebraic structure (we are temporary dropping the heavy notations $(A_1,f_{s_1})$, ...), and define $X_1:=A_1 \cup X_0$, where we suppose that $A_1$ is such that $X_1 \neq X_0$. We clearly have $X \equiv X_0 \subsetneq X_1$. Now consider another algebraic structure $A_2$, and let $X_2:=X_1 \cup A_2$, where we suppose again that $A_2$ is such that $X_2 \neq X_1$. Proceed this way up to, say, $X_k$ (we also allow a countable infinite number of steps, so we could have $k=+\infty$):
\begin{equation}\label{Eq:3.1}
X_n:=X_{n-1} \cup A_n
\end{equation}
with $A_n$ algebraic structure such that $X_n \neq X_{n-1}$ (let $A_0:=X_0 \equiv X$, so that this equation also holds in that case). Clearly, the space $X_k$ constructed this way is a structured space \footnote{To each $x \in X_k$ we associate the first algebraic structure in which they appear: this means that, if $x \in A_n$ and $x \not \in A_j$ $\forall j < n$, then we associate to $x$ the (fixed) neighborhood $A_n$ (recall by Proposition 1.1 in [1] that we can always define a topology on $X_k$ so that all the fixed neighborhoods, in this case all the $A_n$'s, are open; this justifies the possibility of associating $A_n$ to a point $x$ as above).}. The structure map $f_s$ of this space is constructed starting from the chosen structure: for instance, as we said above $A_n$ was just a way to actually say $(A_n,f_{s_n})$. This means that the fixed structure $A_n$ in $X$ gives us $f_s(A_n):=f_{s_n}(A_n)$, for all $n$. Therefore, we have a structured space $(X_k, f_s)$. By definition of $f_s$, we can now only use this map to describe all the local structures in $X_k$, so we will not anymore use $f_{s_n}$ (this is why we dropped the notation with these maps before).\\
We now construct an abelian monoid consisting of the $f_s(A)$'s and their products. First of all, let $\mathcal{U}_n$ indicate the set of all the local structures of $X_n$, that is:
$$ \mathcal{U}_n= \lbrace A_0, A_1, ..., A_n \rbrace $$
As we said above, $f_s|_{X_n}:\mathcal{U}_n \rightarrow \mathcal{T}_n$. In the following, we also need to use the empty set. We thus define:
\begin{equation}\label{Eq:3.2}
f_s(\emptyset):=\lbrace (0), (0), \lbrace \emptyset \rbrace \rbrace
\end{equation}
where $(0)$ means no operations, no properties, respectively, while $\lbrace \emptyset \rbrace$ has already been defined in [1], and indicates that we have no additional non-algebraic structures. Thus, $f_s(\emptyset)$ represents a structure which actually has no structure at all. Let
\begin{equation}\label{Eq:3.3}
\mathcal{\widehat{U}}_n:=\mathcal{U} \cup \lbrace \emptyset \rbrace
\end{equation}
and now consider $\widehat{f}_s: \mathcal{\widehat{U}}_n \rightarrow \mathcal{\widehat{T}}_n$, with the obvious meaning of
\begin{equation}\label{Eq:3.4}
\mathcal{\widehat{T}}_n:=\mathcal{T}_n \cup \lbrace f_s(\emptyset) \rbrace
\end{equation}
Henceforth, we will only use $\widehat{f}_s$; therefore, we will drop this notation, preferring instead the use of $f_s$ (which also avoids confusion with the modified structure map $\widehat{f}_s$ defined in Section 2.5 of [1]; here, however, we will not make use of this modified map, since it is equivalent to the structure map, as shown in [1]). We now define an operation $\times$ for these structures. This is done, in a first step, similarly to what we did in Proposition 3.1 of [1]. More precisely, consider $f_s(A),f_s(B) \in \mathcal{\widehat{T}}_n$. We then define \footnote{Recall by [1] that $f_s(A) \equiv f_s(B)$ means that: (i) $A,B$ are endowed with the same number of operations (they could also be different, as usual in the theory of algebraic structures); (ii) there is a bijection between the operations in $A$ and the ones in $B$ which satisfies the following condition: if $+_r$ is an operation in $A$ to which there corresponds some properties given by some encoding functions in $g_2(A)$ (see [1]), then the associated (by the bijection) operation $\cdot_t$ in $B$ must have precisely the same properties (given this time by encoding functions in $g_2(B)$) of $+_t$ (pay attention to this fact: if $A$ is considered to be a group, and $B$ to be a magma (even though it could be considered a group), we will clearly have that their structures $f_s(A)$ and $f_s(B)$ are different because of (ii)); (iii) they must have the same additional non algebraic structure.}
\begin{equation}\label{Eq:3.5}
f_s(A) \times f_s(B) := \begin{cases} f_s(A) (\equiv f_s(B)), & \text{if}\ f_s(A) \equiv f_s(B) \\ f_s(A), & \text{if}\ f_s(B) \equiv f_s(\emptyset) \\ f_{1,2} (A \times B), & \text{otherwise}\  \end{cases}
\end{equation}
This requires some comments. First of all, $f_{1,2} (A \times B)$ has been defined in the proof of Proposition 3.1 in [1]. Clearly, if $f_s(A) \equiv f_s(B)$, only one (and hence both, since they are the same) is contained in $\mathcal{\widehat{T}}_n$. Thus, this actually means that we are evaluating $f_s(A) \times f_s(A)$. It is well known that products of (Lie) groups are still (Lie) groups, products of rings are still rings, ... Indeed, if instead of using $f_{1,2}$ we define the following $t$-th operation:
$$ (a,b) +_t (c,d) := (a \cdot_t c, b \cdot_t d)$$
for all the (same) operations on $A$ and $B$ (with $f_s(A) \equiv f_s(B)$), we can easily see that we obtain the same number of operations on $A \times B$, which also satisfy the same properties as before. For instance, if $\cdot_t$ is commutative, then $a \cdot_t c =c \cdot_t a$ and $b \cdot_t d = d \cdot_t b$, from which we get $(a,b) +_t (c,d) = (c,d) +_t (a,b)$. Thus also $+_t$ is commutative. This can be easily generalised to many other cases, but it does not always happen. Moreover, in some cases, we need instead to define in another way the new operation on the product space in order to have the same structure. For instance, we do not define componentwise multiplication on $\mathbb{R}^2 \cong \mathbb{C}$ if we want this product of fields to be a field itself. It is clear that, by our definition of equivalence of algebraic structures $\equiv$, we can define somehow a certain kind of operation so that the product of two equivalent structures has still the same structure of the factors. Thus, we will henceforth consider algebraic structures such that $f_s(A) \times f_s(A) \equiv f_s(A)$, where the product structure is defined (\emph{only in this case, where the \textit{two} factors are the same}) in any possible way such that it is equivalent to the structure of each of its factors. See later for a complete definition of 'standard algebraic structures', which will be the only kind of algebraic structures considered in what follows. It is then clear that, in any such case, \eqref{Eq:3.5} is justified even for the product of two equivalent structures. We also explicitely notice that, if $f_s(A) \equiv f_s(B)$, it can be easily verified by \eqref{Eq:3.5} that $f_s(A) \times f_s(C) \equiv f_s(B) \times f_s(C)$.\\
Note that, if $f_s(A)$ is a substructure of $f_s(B)$, their product will still be defined using $f_{1,2}$, because as we explained in [1] we have assigned different structures to them, and we will not (unless we want to change a local structure) switch to the substructure. Moreover, it is clear that if $f_s(A)$ is a proper substructure of $f_s(B)$ ('proper' means that $f_s(A) \not \equiv f_s(B)$) then their product is not equal to $f_s(A)$ (that is, their product does not reduce to the substructure). For instance, it is clear that if $f_s(B)$ has some properties in addition to the ones in $f_s(A)$, then the properties defined component-wise by $f_{1,2}$ are certainly more than the ones in $f_s(A)$. Indeed, just take the properties $(P_j,\widehat{P})$, where $P_j$ is any property of $f_s(A)$ and $\widehat{P}$ is some property of $f_s(B)$ which is not satisfied by $f_s(A)$. Then, these properties cannot reduce to properties of $f_s(A)$ (while, for example, $(P_j,P_j)$ could be reduced to a property ($P_j$) of both $f_s(A)$ and $f_s(B)$). A similar argument also holds for the other components of the structure map, from which we conclude the proof of the statement above.\\
Now we turn to the non-algebraic structures, which can be more difficult to deal with. We are used to structures which maintain the algebraic structures under products; however, our general definition via the structure map does not assure that this is always true. We will call a non-algebraic structure $\lbrace ... \rbrace_C$ 'standard' if it is such that:
\begin{equation}\label{Eq:3.6}
\lbrace ... \rbrace_C \times \lbrace ... \rbrace_C \equiv \lbrace ... \rbrace_C
\end{equation}
which indeed means that the product has still the same non algebraic structure. An algebraic structure is called 'standard' if $f_s(A) \times f_s(A) \equiv f_s(A)$ (which happens quite often with the "most used" algebraic structures, as we have just seen). Henceforth, we will only deal (with this cohomology theory) with standard algebraic structures. With all of this in mind, we can say that $f_s(A) \times f_s(A) \equiv f_s(A)$, from which we conclude that the definition for $f_s(A) \equiv f_s(B)$ given in \eqref{Eq:3.5} makes sense.\\
It is easy to see that the operations and the properties defined in $f_{1,2}$ commute (indeed, since the operations and the properties are defined componentwise in such a way that they are related to either $A$ or $B$, there is no reason why we shoud assume an order); furthermore, if we have $\lbrace ... \rbrace_A \times \lbrace ... \rbrace_B$, it is clear that (for the same reason as above) we can also consider, equivalently, $\lbrace ... \rbrace_B \times \lbrace ... \rbrace_A$. Thus, also the additional non-algebraic structures commute and we can conclude that:
$$ f_s(A) \times f_s(B) \equiv f_s(B) \times f_s(A)$$
For associativity we need to define the product of an element $f_s(A) \times f_s(B)$ (with $f_s(A),f_s(B) \in \mathcal{\widehat{T}}_n$), which belongs to $\mathcal{\widehat{T}}^2 _n:= \lbrace f_s(A) \times f_s(B), \, \text{for some} \, f_s(A),f_s(B) \in \mathcal{\widehat{T}}_n \rbrace$ \footnote{Here $\mathcal{\widehat{T}}_n \times \mathcal{\widehat{T}}_n$ consists of couples $(f_s(A),f_s(B))$, with $f_s(A),f_s(B) \in \mathcal{\widehat{T}}_n$, so it is different from $\mathcal{\widehat{T}}^2 _n$, where instead we used the operation $\times$ defined for the $f_s(A)$'s.} times an element in $\mathcal{\widehat{T}}^2 _n$ (note that this space contains $\mathcal{\widehat{T}}_n$). More generally, the product will be defined starting from this situation. We notice that, in fact, we have define above (up to now) only the product of two element in $\mathcal{\widehat{T}}_n$. The general definition involves the notion of 'minimal product representation'. Consider any product $f_s(A) \times f_s(B)$ with $f_s(A),f_s(B) \in \mathcal{\widehat{T}}_n$. Its minimal product representation $p_m$ is "the simplest structure, decomposed into factor of elements in $\mathcal{\widehat{T}}_n$, to which it is equivalent", i.e.:\\
1) if $f_s(A) \not \equiv f_s(B)$, $p_m(f_s(A) \times f_s(B)):=f_s(A) \times f_s(B)$ \\
2) If $f_s(B)$ is either $\equiv f_s(A)$ or $\equiv f_s(\emptyset)$, then $p_m(f_s(A) \times f_s(B)):=f_s(A)$\\
When dealing with $p_m$, we will NOT use the equivalence relation $\equiv$. This is because of the possible confusion that could arise when defining $\times$ (see below). More precisely, we could also use it, but we then would need to consider another (equivalent) definition of $\times$, which will be shown below for the sake of completeness.\\
More generally, define \textit{inductively}:\\
1) If $f_s(A_1), ..., f_s(A_i)$ are all different from each other, then we define their product $\times$ (which belongs to $\mathcal{\widehat{T}}^i _n \setminus \mathcal{\widehat{T}}^{i-1} _n$) as $f_{1,2}(A_1 \times ... \times A_i$)\\
2) If any structure is equivalent to another one, the minimal product representation will be the product (up to now, this has not been defined yet in this second case, so consider it formally for a while) of the non-equivalent structures (for instance, if $f_s(A) \equiv f_s(B)$ but $\not \equiv f_s(C)$ (and they all belong to $\mathcal{\widehat{T}}_n$), then $p_m(f_s(A) \times f_s(B) \times f_s(C)) = f_s(A) \times f_s(C)$). Then we define the product of the structures as the $f_{1,2}$ whose argument is the product of the sets used as arguments of the structure maps \emph{in the minimal product representation} \footnote{Important: NOT to some structure equivalent to the minimal product representation. This is the reason why we avoided the use of $\equiv$ with $p_m$: we consider only the groups which are arguments of the non-equivalent structures, that is, the ones in the minimal product representation; otherwise, the product would not be well defined. We actually notice that any equivalent expression of the minimal product representation \textit{having the same number of factors} would work as well.} (for instance, in the previous example we would have: $f_s(A) \times f_s(B) \times f_s(C):=f_{1,2}(A \times C)$ (or $f_{1,2}(B \times C)$, which is easily seen to be equivalent)). To say this in other words: consider any possible product representation, and select, among all of them, some factors which are not equivalent to each other. Actually, we will need to consider the collection of all the possible factors which are not equivalent to each other, which can be seen as the 'maximal collection' (these factors are precisely the ones which will be used in the minimal product representation, from which follows the equivalence of these two definitions). Then, the product is given by definition by $f_{1,2}$ evaluated at the product of all the sets used in the maximal collection (and since $f_s(A) \equiv f_s(B)$, and $f_s(C)$ not equivalent to them, where these elements all belong to $\mathcal{\widehat{T}}_n$, imply that $f_s(A) \times f_s(C) =f_{1,2}(A \times C) \equiv f_s(B) \times f_s(C) = f_{1,2}(B \times C)$ (this can be easily generalised), there is no problem with the fact that we can use different sets whose structures are equivalent).\\
We thus finally have a definition of $\times$ for element in $\mathcal{\widehat{T}}^i _n$. Indeed, since any element of $\mathcal{\widehat{T}}^i _n$ can be reduced in minimal product representation, we can always evaluate $f_s(A) \times f_s(B)$, where the two factors belong to some $\mathcal{\widehat{T}}^i _n$'s (even different): their product is defined as the $f_{1,2}$ of the product set of the only sets in the arguments of the structure maps that are factors in the minimal product representation of $f_s(A) \times f_s(B)$, which is clearly given by the product of all the non-equivalent factors of both elements (i.e. we have $p_m(f_s(A) \times f_s(B))=p_m(p_m(f_s(A) \times p_m(f_s(B)))$). We notice that commutativity still holds, by definition of $p_m$ and by definition of $\times$. We also verify that associativity holds: consider the case where in the products there are some equivalent elements (of course, if they are all different there is no problem: the associativity is clear by definition of $f_{1,2}$ by the equivalence relations $\sim_i$, with $i=1,2,3$, defined in Section 2 of [1]; moreover, if there is any $f_s(\emptyset)$, again there is no problem, clearly). Consider $f_s(A) \times (f_s(B) \times f_s(B))$, where the elements in the product belong to some $\mathcal{\widehat{T}}^i _n$'s and with $f_s(A) \not \equiv f_s(B)$ (the other cases can be treated similarly). Then, $f_s(B) \times f_s(B)$ is by definition equivalent to $f_s(B)$ \footnote{It is important to notice that we use $p_m$ \emph{only before} the application of $f_{1,2}$. So, since $f_{1,2}$ can be equivalent to other structure maps, if for instance $f_s(A) \equiv f_s(B) \times f_s(B) \times f_s(C)$, then, by definition (supposing $f_s(B) \not \equiv f_s(C)$ are both in $\mathcal{\widehat{T}}_n$, so that their product is in minimal representation): $f_s(A) \times f_s(A) = f_{1,2}(B \times C)$. However, it is also clear that $(f_s(B) \times f_s(C)) \times (f_s(B) \times f_s(C)) =$ (by definition) $f_{1,2}(B \times C)$ and this is $\equiv  f_s(B) \times f_s(C)$ by \eqref{Eq:3.5}. Thus, we conclude that $f_s(A) = f_{1,2}(B \times C) = f_s(B) \times f_s(C) =: f_s(B) \times f_s(B) \times f_s(C) \equiv f_s(A)$. This can be generalised to more factors simply using the general definition of $\times$ instead of \eqref{Eq:3.5} in the step above. Therefore, since $p_m$ is used only to select the sets that will be used with $f_{1,2}$, we can still use non-minimal representations with the product map $f_{1,2}$ that lead to equivalences as the above one.}, and hence $f_s(A) \times (f_s(B) \times f_s(B)) \equiv f_s(A) \times f_s(B)$. But $(f_s(A) \times f_s(B)) \times f_s(B)$ is defined so that $(f_s(A) \times f_s(B))=f_{1,2}(A_1 \times ... \times A_t \times B_1 \times ... \times B_r)$ (using minimal representations), and consequently $(f_s(A) \times f_s(B)) \times f_s(B)=f_{1,2}(A_1 \times ... \times A_t \times B_1 \times ... \times B_r)$ (because the factors in $f_s(B)$ are all equivalent to some of the ones in the minimal product representation of $f_s(B)$) $ \equiv f_s(A) \times f_s(B)$. We have thus verified that they both give the same result, so that associativity holds. The element $f_s(\emptyset)$ clearly does not change the structure, and thus (this justifies the remaining part of \eqref{Eq:3.5}):
$$f_s(A) \times f_s(\emptyset) \equiv f_s(A)$$
from which we conclude that $f_s(\emptyset)$ is the identity element.\\
We want to have a space which is closed under $\times$. As we said above, $f_s(A) \times f_s(B)$ (with $f_s(A),f_s(B) \in \mathcal{\widehat{T}}_n$) belongs to $\mathcal{\widehat{T}}^2 _n:= \lbrace f_s(A) \times f_s(B), \, \text{for some} \, f_s(A),f_s(B) \in \mathcal{\widehat{T}}_n \rbrace$, which contains $\mathcal{\widehat{T}}_n$ (and, thanks to this inclusion, we can say that $\bigcup_{i=1}^{2} \mathcal{\widehat{T}}^i _n \equiv \mathcal{\widehat{T}}^2 _n$). Thus, more generally (where all the factors below belong to $\mathcal{\widehat{T}}_n$):
\begin{equation}\label{Eq:3.7}
f_s(B_1) \times f_s(B_2) \times ... \times f_s(B_r) \in \bigcup_{i=1}^{r} \mathcal{\widehat{T}}^i _n \equiv \mathcal{\widehat{T}}^r _n
\end{equation}
Then, the space $K_n:=\bigcup_{i=1}^{\infty} \mathcal{\widehat{T}}^i _n$ \footnote{This could also be indicated by $\mathcal{\widehat{T}}^{\infty} _n$, as in the equation \eqref{Eq:3.7}, but here we prefer to use the notation using the infinite union.} is certainly closed under $\times$ (we also note that, since $\mathcal{\widehat{T}}_n \subseteq \mathcal{\widehat{T}}_{n+1}$ $\forall n$, we have $K_n \subseteq K_{n+1}$ $\forall n$). Consequently, by what we said above, this space is an abelian monoid under $\times$. We now could try to use various methods to obtain an abelian group. One of these is the well known Grothendieck completion, used for instance in K-theory, but this would give rise to trivial cohomologies. Another attempt could be made trying to define some inverse elements, say $f_s(A)^{-1}$, which would be a formal structure with the meaning of 'nullifying the action of $f_s(A)$'. Even though this could be rigorously justified, the problem is that our (finite, since it has a finite number of elements) abelian monoid $K_n$ has too many idempotent elements, i.e. elements $a \in K_n$ such that $a=a \times a = a \times a \times a= ...$. The problem is that, it is well known that a finite monoid is a group if and only if it has only one idempotent, the identity element. So, even if we added all the possible inverses, we would still have too many idempotents. We thus have to abandon the idea of constructing an abelian group which contains as subset the abelian monoid $K_n$. The idea is then to use another cohomology theory, Leech cohomology, apply it to all the monoids we have, and then develop a new cohomology theory starting from these ones. This is precisely what we did in Section 2; we now have a justification of the introduction of that general theory.\\
Since all the $K_n$'s are monoids (the fact that they are abelian is not necessary here, because the theory in Section 2 holds for general monoids), we can use such theory in this particular case. We can now state the following important result (which is a simple application of Theorem \ref{Thm:2.1} and Theorem \ref{Thm:2.2}):
\begin{theorem}[Cohomology theory for structured space related to $f_s$]\label{Thm:3.1}
Consider some algebraic structures $A_0, A_1, ..., A_k$ (we also allow a countably infinite number; in such cases, use in what follows $k=\infty$) and define $X_n$ for each $n \geq 0$ as in \eqref{Eq:3.1}, supposing that the $A_n$'s are such that $X_{n+1} \neq X_n$ $\forall n$. Then the spaces $K_n$ defined below \eqref{Eq:3.7} are (abelian) monoids under $\times$ (see \eqref{Eq:3.5} and the general definition below it), and we then define a cohomology theory, which arises from the structure map $f_s$ (by definition of $K_n$), for the structured space $(X_k,f_s)$ as the rectangular cohomology (or the square cohomology, if $k=\infty$) w.r.t. some $\mathcal{A}_n$, $\pi_B$, $\lbrace \widetilde{\partial}(p,j) \rbrace_{p,j \in \mathbb{N}_0}$, as given by Theorem \ref{Thm:2.2} (Theorem \ref{Thm:2.1} if $k=\infty$). If the square/rectangle turns out to be a double complex, then there are two other possible cohomologies that can be defined (always related to $f_s$): these are shown in Section 2.1.
\end{theorem}
\begin{remark}\label{Rm:3.1}
We remark again that we can always choose the trivial family of vertical homomorphisms.
\end{remark}
We notice that, as we wanted, many of the cohomology groups of the new theory given by $f_s$ are often the same as some Leech cohomology groups associated to some (also different) monoids in our sequence. Thus, the interesting part of this cohomology theory is that it sums up various cohomology groups related to different (meaning: related to different monoids) Leech cohomologies, even (possibly) adding other new cohomology groups; choosing a proper path $\pi$, we can encode as much information as possible. This way, our theory has the meaning of 'capturing information from as many Leech cohomologies of the monoids as possible, in order to describe in the most accurate way the given sequence'. As we had already noticed at the beginning of Section 2, \textit{one of the most interesting part of a square/rectangular cohomology} is when we associate a sequence (in this case, the $K_n$'s) to something (in this case, to a structured space), and then we "forget" about the sequence and use the cohomology theory for the object considered (\textit{so, after applying one of the above results, we can forget about the $K_n$'s and just use the cohomology theory obtained for the structured space}).\\
We conclude observing that, the more the different structures $f_s(A)$'s, the more interesting the cohomology theory is. For instance, a $\mu$-CRD (see Section 4 in [1]) is really interesting, in particular if its $\mathcal{U}$ contains a countably infinite (and not finite) number of $U_p$'s.
%%%%%%%%%%%%%%%%%%%%%%
\section{A cohomology theory related to $h$}
We conclude this article with a cohomology theory which arises from $h$ (see Section 4 in [1]). Again, this will follow from a simple application of the results in Section 2 (in this case, of Theorem \ref{Thm:2.2}), but before of this we need to find some monoids related to this function. Consider a structured space $X$ with $\mathcal{U}=\lbrace U_p \rbrace_{p=0}^{k}$ (with $k \neq \infty$; obviously, the $U_p$'s in $\mathcal{U}$ are all different (i.e. $U_j \neq U_i$ $\forall i \neq j$), since it is a collection and so it does not contain copies of the same element), and suppose that $h$ is surjective onto the power collection without the empty sets (as in the hypothesis of Theorem 4.1 in [1]). To have an idea of the meaning of this last assumption, note that the surjectivity of $h$ means that
$$\bigcap_{U_p \in \mathcal{B}} U_p \neq \emptyset$$
whichever is the nonempty subcollection $\mathcal{B}$ of $\mathcal{U}$. We will need the following equivalence relations:
\begin{equation}\label{Eq:4.1}
x \sim_1 y \Leftrightarrow h(x)=h(y)
\end{equation}
and
\begin{equation}\label{Eq:4.2}
h(x) \sim h(y) \Leftrightarrow |h(x)|=|h(y)|
\end{equation}
where $|h(x)|$ is the cardinality of the collection \footnote{For instance, if $\mathcal{C}=\lbrace A, B, C \rbrace$ (of course, with $A \neq B$, $B \neq C$ and $A \neq C$), $|\mathcal{C}|=3$.}, which is always finite since $k \neq \infty$ and the image of $h$ is the power collection (cardinality: $2^k$) without the empty sets. We choose the "representatives" of the equivalence classes $[h(x)]$ as follows: by the surjectivity of $h$, by \eqref{Eq:4.1} and by \eqref{Eq:4.2}, we know that we can fix a $U_t \in \mathcal{U}$ such that the representatives of all these equivalence classes contain it. Then, we can fix another (different) $U_m \in \mathcal{U}$ and choose the representatives so that all of them except one and only one of them contain it. Proceeding this way, we will get one and only one representative (call it $h_0$) containing only one element, one and only one representative (call it $h_1$) containing only two elements (and one of them is the element contained in $h_0$), ... After properly renaming the $U_p$'s, we can write:
\begin{equation}\label{Eq:4.3}
h_r=\lbrace U_0, U_1, ..., U_r \rbrace
\end{equation}
It is then clear that we have a chain of inclusions:
\begin{equation}\label{Eq:4.4}
h_0 \subsetneq h_1 \subsetneq h_2 \subsetneq ...
\end{equation}
Now, if we define
\begin{equation}\label{Eq:4.5}
\widetilde{h}_r:=h_r \, \cup \, \lbrace \emptyset \rbrace
\end{equation}
it is clear that \eqref{Eq:4.4} still holds replacing each $h_r$ with $\widetilde{h}_r$. Consider the spaces:
\begin{equation}\label{Eq:4.6}
g_r:= \lbrace \bigcup_{A \in \mathcal{A}, \, \text{for some} \, \mathcal{A} \subseteq \widetilde{h}_r } A \rbrace
\end{equation}
(i.e. $g_r$ contains all the possible unions of elements in $\widetilde{h}_r$). Clearly, $g_r$ is closed under $\cup$; furthermore $\cup$ is commutative and associative in $g_r$, and there is also an identity element ($\emptyset$) in this space. Thus, each $g_r$ is an abelian monoid. As in the previous section, Grothendieck group construction would lead to trivial cohomologies in such a situation, and again adding some kind of 'inverse elements' would not be useful, because we have too many idempotent elements (in fact, all the elements are idempotent under $\cup$, since obviously $A \cup A \cup ... =A$ for any $A$). We thus need some other construction for our cohomology theory; the rectangular cohomology developed in Section 2 is an interesting answer to this problem (since $k$ is finite, it is clear that the chain of inclusions with the $\widetilde{h}_r$'s is finite. Since that chain still holds replacing each $\widetilde{h}_r$ with $g_r$ (as it can be easily proved), we will have a finite number of monoids). A simple application of Theorem \ref{Thm:2.2} in this situation yelds the following important result:
\begin{theorem}\label{Thm:4.1}
Let $(X,f_s)$ be a structured space with $\mathcal{U}=\lbrace U_p \rbrace_{p=0}^{k}$, $k \neq \infty$. Suppose that the map $h$ defined in Section 4 of [1] is surjective onto the power collection without the empty sets. Then use this assumption together with \eqref{Eq:4.1} and \eqref{Eq:4.2} to reorder the $U_p$'s in such a way that \eqref{Eq:4.3} holds. Define $g_r$ as in \eqref{Eq:4.6}; each $g_r$ is an abelian monoid under $\cup$, so we can consider $\lbrace g_r \rbrace_{r=0}^{\widehat{k}}$ (where $\widehat{k}$ is the number of $g_r$'s minus $1$, which is clearly finite) and apply Theorem \ref{Thm:2.2} to get a cohomology theory for structured spaces that arises from $h$. If the rectangle turns out to be a double complex, then there are two other possible cohomologies that can be defined (always related to $h$): these are shown in Section 2.1.
\end{theorem}
%%%%%%%%%%%%%%
\section{Conclusion}
As we had noticed in the conclusion of [1], the maps $f_s$ and $h$ have some relations with presheaves, vector bundles, ... This suggested the existence of some cohomology theories for structured spaces related to these two functions. In this paper we have indeed provided such theories, starting from a general cohomology which is obtained using paths on a square or a rectangle. We notice that the square and rectangular cohomologies can be applied in many other situations, as explained for instance in Remark \ref{Rm:2.1}. This general method is particullary interesting when we associate a sequence to some object that we want to study, and we then construct square/rectangular cohomology starting from it (we can also "forget" about the sequence after this passage). Moreover, it is often useful to choose paths that give rise to long local cochains, as explained in Section 2.2, in order to assure that Eilenberg-Steenrod axioms locally hold.\\
\\
\begin{large}
\textbf{Conflict of interest}
\end{large}
\\
The author declares that he has no conflict of interest.

\end{document}